\documentclass[11pt]{article}

\pdfoutput=1

\usepackage{url}
\usepackage{fullpage}
\usepackage{amsmath,amsthm}
\usepackage{euler}
\usepackage{amssymb}
\usepackage{amscd}
\usepackage{latexsym}
\usepackage{booktabs}
\usepackage{tabularx}
\usepackage{graphicx}   
\usepackage{subfigure}
\usepackage{hyperref}
\usepackage{enumerate}

\newtheorem{definition}{Definition}
\newtheorem{theorem}{Theorem}

\newtheorem{proposition}{Proposition}
\newtheorem{lemma}{Lemma}
\newtheorem{remark}{Remark}

\title{Learning from Survey Training Samples:\\
Rate Bounds for Horvitz-Thompson Risk Minimizers}

\author{
Stephan Cl\'emen\c{c}on\thanks{LTCI, T\'el\'ecom ParisTech, Universit\'e Paris-Saclay, 75013, Paris, France. \texttt{\{stephan.clemencon,guillaume.papa\}@telecom-paristech.fr}}~, 
Patrice Bertail\thanks{Universit\'e Paris Nanterre,
MODALX, Nanterre, France. \texttt{patrice.bertail@u-paris10.fr}}~, 
Guillaume Papa\footnotemark[1]
}

\begin{document}
\maketitle

\begin{abstract}
 The generalization ability of minimizers of the empirical risk in the context of binary classification has been investigated under a wide variety of complexity assumptions for the collection of classifiers over which optimization is performed. In contrast, the vast majority of the works dedicated to this issue stipulate that the training dataset used to compute the empirical risk functional is composed of i.i.d. observations and involve sharp control of uniform deviation of i.i.d. averages from their expectation. Beyond the cases where training data are drawn uniformly without replacement among a large i.i.d. sample or modelled as a realization of a weakly dependent sequence of r.v.'s, statistical guarantees when the data used to train a classifier are drawn by means of a more general sampling/survey scheme and exhibit a complex dependence structure have not been documented in the literature yet. It is the main purpose of this paper to show that the theory of empirical risk minimization can
 be extended to situations where statistical learning is based on survey samples and knowledge of the related (first order) inclusion probabilities. Precisely, we prove that minimizing a (possibly biased) weighted version of the empirical risk, refered to as the (approximate) Horvitz-Thompson risk (HT risk), over a class of controlled complexity lead to a rate for the excess risk of the order $O_{\mathbb{P}}((\kappa_N (\log N)/n)^{1/2})$ with $\kappa_N=(n/N)/\min_{i\leq N}\pi_i$, when data are sampled by means of a rejective scheme of (deterministic) size $n$ within a statistical population of cardinality $N\geq n$, a generalization of basic {\it sampling without replacement} with unequal probability weights $\pi_i>0$. Extension to other sampling schemes are then established by a coupling argument. Beyond theoretical results, numerical experiments are displayed in order to show the relevance of HT risk minimization and that ignoring the sampling scheme used to generate the training dataset may completely jeopardize the learning procedure.
\end{abstract}

\section{Introduction}
\label{sec:intro}
 Whereas statistical learning techniques crucially exploit data that can serve as examples to train a decision rule, they may also make use of weights individually assigned to the observations,
resulting from survey sampling stratification. Such weights could correspond either
to true inclusion probabilities or else to calibrated or post-stratification
weights, minimizing some discrepancy under certain margin constraints for the
inclusion probabilities. In the context of statistical inference based on survey data, the asymptotic properties of specific statistics such as Horvitz-Thompson estimators (\textit{cf} \cite{HT51}), whose computation involves not only the observations but also the weights, have been widely  investigated: in particular, mean estimation and
regression have been the subject of much attention, refer to
\cite{Hajek64}, 
 \cite{Rob82}, 
  \cite{Ber98} for
instance, and a
comprehensive functional limit theory for distribution function estimation is progressively documented in the statistical literature,
see 
 \cite{Breslow06}, 
\cite{Breslow09}, \cite{Wellner11}. 
At the same time, the last decades have witnessed a rapid development of the field of machine-learning.
Revitalized by different breakout algorithms (\textit{e.g.} SVM,
boosting methods), its practice is now supported by a sound probabilistic theory based on
recent non asymptotic results in the study of empirical processes, see
 \cite{Kolt06}, \cite{BBL05}. However, most papers dedicated to theoretical results grounding the \textit{Empirical Risk Minimization} approach (ERM in short), the main paradigm of statistical learning, assume that the training of a decision rule is based on a dataset formed of independent replications of a generic random vector $Z$, a collection of $N\geq 1$ i.i.d. observations $Z_1,\; \ldots,\; Z_N$ namely. In contrast, few results are available in situations where the training dataset is generated by a more complex sampling scheme. One may refer to \cite{BardenetMaillard} for concentration inequalities permitting to study the generalization ability of empirical risk minimizers when the training data are obtained by standard \textit{sampling without replacement} (SWOR in abbreviated form) or to \cite{Steinwart09} in the case where the decision rule is learnt from a path of a \textit{weakly dependent stochastic} process.

It is the goal of this paper to extend the ERM theory to situations where the training dataset is generated by means of a more general sampling scheme, with possibly unequal probability weights. We first consider the case of \textit{rejective} sampling (sometimes refered to as \textit{conditional Poisson} sampling), an important generalization of basic SWOR. The rate bound results obtained by means of properties of so-termed \textit{negatively associated random variables} in this case are next shown to extend to a class of more general sampling schemes by a coupling argument. In addition, numerical experiments have been carried out in order to provide empirical evidence of the approach developed. They show in particular that statistical accuracy of the ERM approach may go down the drain if the sampling scheme underlying the training dataset is ignored.

The paper is organized as follows. In section \ref{sec:background}, the probabilistic framework of the present study is described at length and basic results of the probabilistic theory of classification are briefly recalled, together with some important notions of survey theory. The main theoretical results are stated in section \ref{sec:main}, while illustrative numerical experiments are presented in section \ref{sec:num}. Certain proofs are sketched in the Appendix section, whereas additional technical details are deferred to the Supplementary Material.

\section{Background and Preliminaries}\label{sec:background}

As a first go, we start with recalling key concepts pertaining to the theory of empirical risk minimization in binary classification, the flagship problem in statistical learning. A few notions related to survey theory are next described, which will be involved in the subsequent analysis.
 Throughout the article, the indicator function of any event
$\mathcal{E}$ is denoted by $\mathbb{I}\{\mathcal{E}\}$, the Dirac mass at any
point $a$ by $\delta_{a}$, the power set of any set $E$ by
$\mathcal{P}(E)$, the cardinality of any finite set $A$ by $\#A$.

\subsection{Binary Classification - Empirical Risk Minimization Theory}

The binary classification problem is considered as a running example all along
the paper. Because it can be easily formulated, it is undeniably the most documented
statistical learning problem in the literature and certain results extend to
more general frameworks (\textit{e.g.} multiclass classification, regression,
ranking). Let $(\Omega, \mathcal{A},\mathbb{P})$ be a probability space and $(X,Y)$ a random pair defined on $(\Omega,
\mathcal{A},\mathbb{P})$, taking its values in some measurable product space
$\mathcal{X}\times\{-1,+1\}$, with common distribution $P(dx, dy)$: the r.v. $X$ models some observation, hopefully useful for predicting the binary label $Y$. The distribution
$P$ can also be described by the pair $(F, \eta)$ where $F(dx)$ denotes the
marginal distribution of the input variable $X$ and $\eta(x)=\mathbb{P}%
\{Y=+1\mid X=x \}$, $x\in\mathcal{X}$, is the \textit{posterior distribution}. The objective is to build, based on the training dataset at disposal, a
measurable mapping $g:\mathcal{X} \mapsto\{-1,+1\}$, called a
\textit{classifier}, with minimum risk:
\begin{equation}
\label{eq:risk}L(g)\overset{def}{=}\mathbb{P}\{g(X)\neq Y\}.
\end{equation}
It is well-known folklore in the probabilistic theory of pattern recognition
that the \textit{Bayes classifier} $g^{*}(x)=2\mathbb{I}\{\eta(x)\geq 1/2\}-1$ is
a solution of the risk minimization problem $\inf_{g} L(g)$, where the infimum
is taken over the collection of all classifiers defined on the input space
$\mathcal{X}$. The minimum risk is denoted by $L^{*}=L(g^{*})$. Since the
distribution $P$ of the data is unknown, one substitutes the true risk with
its empirical estimate
\begin{equation}
\label{eq:emp_risk}\widehat{L}_{n}(g)=\frac{1}{n}\sum_{i=1}^{n}\mathbb{I}%
\{g(X_{i})\neq Y_{i}\},
\end{equation}
based on a sample $(X_1,Y_1),\; \ldots,\; (X_n,Y_n)$ of independent copies of the generic random pair $(X,Y)$.
The true risk minimization is then replaced by the empirical risk minimization
\begin{equation}
\label{eq:ERM}\min_{g\in\mathcal{G}}\widehat{L}_{n}(g),
\end{equation}
where the minimum is taken over a class $\mathcal{G}$ of classifier
candidates, supposed rich enough to include the naive Bayes classifier (or a
reasonable approximation of the latter). Considering a solution $\widehat
{g}_{n}$ of \eqref{eq:ERM}, a major problem in statistical learning theory is
to establish upper confidence bounds on the \textit{excess of risk}
$L(\widehat{g}_{n})-L^{*}$ in absence of any distributional assumptions and
taking into account the complexity of the class $\mathcal{G}$ (\textit{e.g.}
described by geometric or combinatorial features such as the \textsc{VC}
dimension) and some measure of accuracy of approximation of $P$ by its
empirical counterpart $P_{n}=(1/n)\sum_{i=1}^{n}\delta_{(X_{i},Y_{i})}$ over
the class $\mathcal{G}$. Indeed, one typically bounds the excess of risk of
the empirical risk minimizers as follows
\[
L(\widehat{g}_{n})-L^{*}\leq2\sup_{g\in\mathcal{G}}\vert\widehat{L}%
_{n}(g)-L(g)\vert+\left( \inf_{g\in\mathcal{G}}L(g)-L^{*} \right) .
\]
The second term on the right hand side is referred to as the \textit{bias} and
depends on the richness of the class $\mathcal{G}$, while the first term,
called the \textit{stochastic error}, is controlled by means of results in
empirical process theory, see \cite{BBL05}.
\begin{remark}{\sc (On risk surrogates)}\label{remarkSurrogate} Although its study is of major interest from a theoretical perspective, the problem \eqref{eq:ERM} is generally NP-hard. For this reason, the cost function $\mathbb{I}\{-Yg(X)>0\}$ is replaced in practice by a nonnegative convex cost function $\phi(Yg(X))$, turning empirical risk minimization to a tractable convex optimization problem. Typical choices include the exponential cost $\phi(u)=\exp(u)$ used in boosting algorithms, the hinge loss $\phi(u)=(1+u)_+$ in the case of support vector machines and the \textit{logit} cost $\phi(u)=\log(1+\exp(u))$ for Neural Networks, see \cite{BJM06} and the references therein. Extension of the results established in the present paper to such risk surrogates are straightforward and left to the reader.
\end{remark}
In this paper, we consider the situation where the training data used to compute of the
empirical risk \eqref{eq:emp_risk} is not an i.i.d. sample but the product of a more general sampling plan of fixed size $n\geq 1$.

\subsection{Sampling Schemes and Horvitz-Thompson Estimation}\label{subsec:survey}
Let $N\geq1$. In the standard \textit{superpopulation} framework we consider, $(X_{1},Y_{1}),\;\ldots,\;
(X_{N},Y_{N})$ is a sample of independent copies of $(X,Y)$
observed on a finite population $\mathcal{I}_{N} := \{1,\; \ldots,\; N\}$.
We call a \textit{survey sample} of (possibly random) size $n \leq N$ of the
population $\mathcal{I}_{N}$, any subset $s := \{i_{1}, \dots, i_{n(s)}\}
\in\mathcal{P}(\mathcal{I}_{N})$ with cardinality $n =: n(s)$ less that $N$. A
sampling design without replacement is determined by a conditional probability
distribution $R_{N}$ on the set of all possible samples $s \in\mathcal{P}%
(\mathcal{I}_{N})$ given the original data $\mathcal{D}_{N}=\{(X_{i},Y_{i}):\;
i\in\mathcal{I}_{N}\}$. For any $i\in\{1,\; \ldots,\; N\}$, the first order
\textit{inclusion probability}, $\pi_{i}= \mathbb{P}_{  R_{N}}\{i \in S\}$ is the probability that the unit $i$ belongs to a random
sample $S$ drawn from the conditional distribution $R_{N}$. We set
$\boldsymbol{\pi}= (\pi_{1},\; \ldots,\; \pi_{N})$. The second order inclusion
probabilities are denoted by $\pi_{i,j}= \mathbb{P}_{R_{N}}\{(i,j)\in S^{2}\}$, for any $i\neq j$ in $\{1,\dots,N\}^{2}$. The
information related to the observed sample $S \subset\{1,\dots, N\}$ is fully
enclosed in the r.v. $\boldsymbol{\epsilon}_{N} = (\epsilon_{1},\; \ldots
,\; \epsilon_{N})$, where $\epsilon_{i} = \mathbb{I}\{i \in S\}$ for $1\leq i
\leq N$. The $1$-d marginal conditional distributions of the sampling scheme
$\boldsymbol{\epsilon}_{N}$ given $\mathcal{D}_N$ are the Bernoulli distributions $\mathcal{B}%
(\pi_{i})=\pi_{i} \delta_{1} +(1-\pi_{i})\delta_{0}$, $1\leq i\leq N$, and the
covariance matrix $\Gamma_{N}$ of the r.v. $\boldsymbol{\epsilon}_{N}$ has
entries given by $\Gamma_{N}(i,j)= \pi_{i,j} - \pi_{i}\pi_{j}$, with
$\pi_{i,i}=\pi_{i}$ by convention, for $1\leq i, j\leq N$. Observe that,
equipped with the notations above, $\sum_{1\leq i\leq N}\epsilon_{i} = n(S)$. One
may refer to 
\cite{Dev87} for accounts of survey sampling techniques. Notice also that, in many applications, the inclusion probabilities are built using some extra information, typically by means of {\it auxiliary random variables} $W_1,\; \ldots,\; W_N$ defined on $(\Omega, \mathcal{A},\mathbb{P})$ and taking their values in some measurable space $\mathcal{W}$:  $\forall i\in \{1,\; \ldots,\; N\}$,
   $\pi_i=nh(W_i)/\sum_{1\leq j\leq N}h(W_j)$,
   where $n\max_{1\leq i \leq n}h(W_i)\leq \sum_{1\leq i\leq N}h(W_i)$ almost-surely and $h:\mathcal{W}\rightarrow ]0,\; +\infty[$ is a measurable \textit{link function}.
    The $(X_i,Y_i, W_i)$'s are generally supposed to be i.i.d. copies of a generic r.v. $(X,Y,W)$. See \cite{Sarndall03} for more details. For simplicity, the $\pi_i$'s are supposed to be deterministic in the subsequent analysis, which boils down to carrying out the study conditionally upon the $W_i$'s in the example aforementioned.
    
 {\bf Horvitz-Thompson risk.}   As defined in \cite{HT51}, the Horvitz-Thompson version of the (not available) empirical
    risk $\widehat{L}_{N}(g) = N^{-1}\sum_{1\leq i\leq N} \mathbb{I}\{Y_i\neq g(X_i) \}$ of any classifier candidate $g$ based
    on the sampled data $\{(X_i,Y_i):\; i\in S  \}$ with $S\sim R_N$ is given by:
   \begin{eqnarray}
       \label{eq:HT_risk}\overline{L}_{\boldsymbol{\epsilon}_{N}}(g) = \frac{1}{N}\sum_{i\in S}\frac{1}{\pi_i}\mathbb{I}\{g(X_i)\neq Y_i \} 
       =\frac{1}{N}\sum_{i=1}^{N}\frac{\epsilon_{i}}{\pi_{i}} \, \mathbb{I}\left\{
       g(X_{i}) \neq Y_{i} \right\}
       \end{eqnarray}
    with the convention that $0/0=0$ and where the subscript $\boldsymbol{\epsilon
    }_{N}= (\epsilon_{1}, \dots, \epsilon_{N})$ denotes the vector in
    correspondence with the sample $S$. Observe that, conditionally upon the $(X_{i}, Y_{i})$'s, the
    quantity \eqref{eq:HT_risk}, that shall be referred to as the \textit{empirical Horvitz-Thompson risk}
        (empirical HT risk in short) throughout the paper, is an unbiased estimate of the empirical risk $\widehat{L}_{N}(g)$. Its
    (pointwise) consistency and asymptotic normality are established in
    \cite{Rob82} and \cite{Ber98} for a variety of sampling schemes. 
    
    This article is devoted to investigating the statistical performance of
    minimizers $\bar{g}_{N}$ of the HT risk \eqref{eq:HT_risk} over the class $\mathcal{G}$
    under adequate assumptions for the sampling scheme $R_N$ used to generate the training dataset. 
    We point out that such an analysis is far from straightforward due to the possible depence structure of the terms involved in the summation \eqref{eq:HT_risk}: except in the Poisson case (recalled below), concentration results for empirical processes cannot be directly applied to control maximal deviations of the type $$\sup_{g\in \mathcal{G}}\vert \overline{L}_{\boldsymbol{\epsilon}%
    _{N}}(g) - L(g) \vert.$$

 {\bf Conditional Poisson sampling.} One of the simplest sampling plan is undeniably the \textit{Poisson survey scheme} (without replacement), a generalization of \textit{Bernoulli sampling} originally proposed in \cite{Goodman} for the case of unequal weights: the $\epsilon_i$'s are independent and the sampling distribution is thus entirely determined by the first order inclusion probabilities $\mathbf{p}_N=(p_1,\; \ldots, \; p_N)\in ]0,1[^N$:
   \begin{equation}\label{eq:Poisson}
   \forall s\in \mathcal{P}(\mathcal{I}_N),\;\; P_N(s)=\prod_{i\in S}p_i\prod_{i\notin S}(1-p_i).
   \end{equation}

    Observe in addition that the behavior of the quantity \eqref{eq:HT_risk} can be then investigated by means of results established for sums of independent random variables. However, the major drawback of this sampling plan lies in the random nature of the corresponding sample size, impacting significantly the variability of \eqref{eq:HT_risk}. The variance of the Poisson sample size is given by $d_N=\sum_{i=1}^N p_i(1-p_i)$, while the conditional variance of \eqref{eq:HT_risk} is in this case:
    $\sum_{i=1}^n((1-p_i)/p_i)\mathbb{I}\{g(X_i)\neq Y_i\}$.
 For this reason, \textit{rejective sampling}, a sampling design $R_N$ of fixed size $n\leq N$, is often preferred in practice. It generalizes the \textit{simple random sampling without replacement} (where all samples with cardinality $n$ are equally likely to be chosen, with probability $(N-n)!/n!$, all the corresponding first and second order probabilities being thus equal to $n/N$ and $n(n-1)/(N(N-1))$ respectively). Denoting by $\boldsymbol{\pi}_N=(\pi_1,\; \ldots,\; \pi_N)$ its first order inclusion probabilities and by $\mathcal{S}_n=\{s\in \mathcal{P}(\mathcal{I}_N):\; \#s=n  \}$ the subset of all possible samples of size $n$, it is defined by:
 \begin{equation}\label{eq:Rejective}
\forall s\in \mathcal{S}_n,\;\;  R_N(s)=C \prod_{i\in s}p_i \prod_{i\notin s}(1-p_i),
 \end{equation}
 where $C=1/\sum_{s\in \mathcal{S}_n}\prod_{i\in s}p_i \prod_{i\notin s}(1-p_i)$ and the vector $\mathbf{p}_N=(p_1,\; \ldots,\; p_N)\in]0,1[^N$ yields first order inclusion probabilities equal to the $\pi_i$'s and is such that  $\sum_{i\leq N}p_i=n$. Under this latter additional condition, such a vector $\mathbf{p}_N$ exists and is unique (see \cite{Dupacova}) and the related representation \eqref{eq:Rejective} is then said to be \textit{canonical}\footnote{Notice that any vector $\mathbf{p}'_N\in ]0,1[^N$ such that $p_i/(1-p_i)=cp'_i/(1-p'_i)$ for all $i\in\{1,\; \ldots,\; n\}$ for some constant $c>0$ can be used to write a representation of $R_N$ of the same type as \eqref{eq:Rejective}}.  Comparing \eqref{eq:Rejective} and \eqref{eq:Poisson} reveals that rejective $R_N$ sampling of fixed size $n$ can be viewed as Poisson sampling given that the sample size is equal to $n$. It is for this reason that rejective sampling is usually referred to as \textit{conditional Poisson sampling}. 
 One must pay attention not to get the $\pi_i$'s and the $p_i$'s mixed up: the latter are the first order inclusion probabilities of $P_N$, whereas the former are those of its conditional version $R_N$. However they can be related by means of the results stated in \cite{Hajek64} (see Theorem 5.1 therein): $\forall i\in \{1,\; \ldots,\; N  \}$,
 \begin{eqnarray}\label{eq:rel1}
 \pi_i(1-p_i)&=&p_i(1-\pi_i)\times ( 1-( \tilde{\pi}-\pi_i )/d^*_N  +o(1/d^*_N)  ,\label{eq:rel1}\\\label{eq:rel2}
  p_i(1-\pi_i)&=&\pi_i(1-p_i)\times ( 1-( \tilde{p}-p_i )/d_N +o(1/d_N)  ,\label{eq:rel2}
 \end{eqnarray}
 where $d^*_N=\sum_{i=1}^N\pi_i(1-\pi_i)$, $\tilde{\pi}=(1/d^*_N)\sum_{i=1}^N\pi_i^2(1-\pi_i)$ and $\tilde{p}=(1/d_N)\sum_{i=1}^N(p_i)^2(1-p_i)$.
 
More examples of sampling schemes with fixed size are given in the Supplementary Material. 
 of survey theory.

\section{Main Results}\label{sec:main}

We first consider the case where statistical learning is based on the observation of a sample drawn by means of a rejective scheme. As shall be seen below, the main argument underlying the results obtained relies on the fact that the related scheme form a collection of \textit{negatively associated} (binary) random variables, a rather tractable type of dependence structure. This property being shared by many other sampling schemes of deterministic size, the same argument can be thus naturally applied to carry out a similar rate analysis for training data produced by such plans. Extensions of these results to more general sampling schemes are also considered by means of a \textit{coupling} technique.

\subsection {Horvitz-Thompson Empirical Risk Minimization in the  Rejective Case}

For clarity, we first recall the definition of {\it negatively associated random variables}, see \cite{JDP83}.
 \begin{definition}\label{def:negassoc}
Let $Z_1,\; \ldots,\; Z_n$ be random variables defined on the same probability space, valued in a measurable space $(E,\mathcal{E})$. They are said to be negatively associated iff 
 for any pair of disjoint subsets $A_1$ and $A_2$ of the index set $\{1,\; \ldots,\; n  \}$
 \begin{equation}\label{eq:neg}
 Cov \left( f((Z_i)_{i\in A_1}),\; g((Z_j)_{j\in A_2}) \right)\leq 0,
 \end{equation}
 for any real valued measurable functions $f:E^{\#A_1}\rightarrow \mathbb{R}$ and $g:E^{\#A_2}\rightarrow \mathbb{R}$ that are both increasing in each variable.
 \end{definition}

The theorem stated below reveals that any rejective scheme $\boldsymbol{\epsilon}_N$ forms a collection of negatively associated r.v.'s. The proof is given in the Appendix section.
\begin{theorem}\label{thm:neg}
Let $N\geq 1$ and $\boldsymbol{\epsilon}_N=(\epsilon_1,\; \ldots,\; \epsilon_N)$ be the vector of indicator variables related to a rejective plan on $\mathcal{I}_N$. Then, the binary random variables $\epsilon_1,\; \ldots,\; \epsilon_N$ are negatively associated.
\end{theorem}

The result above permits to handle the dependence of the terms involved in the summation \eqref{eq:HT_risk}. It is the key argument for proving the following proposition, which extends results for training datasets generated by basic sampling without replacement (\textit{i.e.} in the case of all equal weights: $\pi_i=n/N$ for $i=1,\; \ldots,\; N$), refer to \cite{BardenetMaillard} (see also \cite{Serfling74}).

\begin{proposition}\label{prop:rate1}
Suppose that the sampling scheme $\boldsymbol{\epsilon}_N$ is rejective with first order inclusion probabilities $\boldsymbol{\pi}_N$ and that the class $\mathcal{G}$ is of finite {\sc VC} dimension $V<+\infty$. Set $\kappa_N=(n/N)/\min_{i\leq N}\pi_i$. Then, the following assertions hold true.
\begin{itemize}
\item [(i)] For any $\delta\in (0,1)$, with probability larger than $1-\delta$, we have: $\forall n\leq N$,
\begin{align}
\sup_{g\in \mathcal{G}}\vert \bar{L}_{\boldsymbol{\epsilon}_N}(g)-\widehat{L}_N(g) \vert \leq 2\kappa_N\frac{\log(\frac{2}{\delta})+V\log (N+1)}{3n}
 +\sqrt{2\kappa_N\frac{\log(\frac{2}{\delta})+V\log (N+1)}{n}}.
\end{align}
\item[(ii)] For any solution $\bar{g}_{N}$ of the minimization problem $\inf_{g\in \mathcal{G}}\overline{L}_{\boldsymbol{\epsilon}_N}(g)$ is such that, for any $\delta\in (0,1)$, with probability at least $1-\delta$, we have: $\forall N\geq 1$,
\begin{align*}
L(\bar{g}_N)-L^*&\leq
2\sqrt{2\kappa_N\frac{\log(\frac{4}{\delta})+V\log (N+1)}{n}}
+4\kappa_N\frac{\log(\frac{4}{\delta})+V\log (N+1)}{3n}\\
&+C\sqrt{\frac{V}{N}}+2\sqrt{\frac{2\log (\frac{2}{\delta})}{N}}+ \inf_{g\in \mathcal{G}}L(g)-L^*.
\end{align*}
\end{itemize}
\end{proposition}
The factor $\kappa_N$ involved in the bounds above reflects the influence of the sampling scheme (notice incidentally that $\kappa_N\geq1$ since $\sum_{i\leq N}\pi_i=n$). In the SWOR case, \textit{i.e.} when $\pi_i=n/N$ for all $i\in \{1,\; \ldots,\; N\}$, it is then minimum, equal to $1$. More generally, when $n=o(N)$ as $N\rightarrow +\infty$, as soon as the weights cannot vanish faster than $n/N$, the rate achieved by minimizers of the HT risk is of the order $O_{\mathbb{P}}(\sqrt{(\log N)/n})$.
Many sampling schemes (\textit{e.g.} Rao-Sampford sampling, Pareto sampling, Srinivasan sampling) of fixed size are actually described by random vectors $\boldsymbol{\epsilon}_N$ with negatively associated components, see \cite{BJ12} or \cite{KCR11}. Hence, a rapid examination of Proposition \ref{prop:rate1}'s proof shows that the bounds stated above immediately extend to these cases. See the Supplementary Material for more details and references. Before showing how the rate bounds established can be extended to even more general sampling schemes, a few remarks are in order.
\begin{remark}{\sc (Complexity assumptions)} We point out that the results stated can be established, essentially by means of the same argument as that developed in the Appendix, under complexity assumptions of different nature, involving metric entropy conditions for instance (see \textit{e.g.} \cite{VVW96}). Such straightforward extensions are left to the reader.
\end{remark}
\begin{remark}{\sc (Model Selection)}
A slight modification of the argument involved in Proposition \ref{prop:rate1} straightforwardly leads to bounds on the expected excess risk $\mathbb{E}[L(\bar{g}_{\boldsymbol{\epsilon}_N}) ]-\inf_{g\in \mathcal{G}}L(g)$. Following the \textit{Structural Risk Minimization} principle (see \cite{Vapnik}), such {\sc VC} bounds can be next used as complexity regularization terms to penalize additively the HT risk \eqref{eq:HT_risk} and, for a sequence of model classes $\mathcal{G}_k$ with $k\geq 1$ of finite {\sc VC} dimension, select the classifier among the minimizers $\{ \arg \min_{g\in \mathcal{G}_k}\bar{L}_{\boldsymbol{\epsilon}_N}(g),\; k\geq 1\}$, which has approximately minimal risk. Due to space limitations, details are left to the reader.
\end{remark}
\begin{remark} {\sc (Biased HT risk)} As recalled in the Supplementary Material, the canonical parameters $\mathbf{p}_N $ are practically used to build a rejective sampling scheme $\boldsymbol{\epsilon}_N$ rather than its vector of first order inclusion probabilities $(\pi_1,\; \ldots,\; \pi_N)$, whose explicit computation based on the $p_i$'s is a difficult task, refer to \cite{CDL94} for dedicated algorithms. For this reason, one could be naturally tempted to minimize the alternative risk estimate
$\widetilde{L}_{\boldsymbol{\epsilon}_N}(g)=(1/N)\sum_{i\leq N}(\epsilon_i/p_i)\mathbb{I}\{ Y_i\neq g(X_i) \}.$
As proved in the Supplementary Material, refinements of Eq. \eqref{eq:rel1}-\eqref{eq:rel2} show that 
\begin{align}\label{eq:bias}
\sup_{g\in \mathcal{G}}\vert \widetilde{L}_{\boldsymbol{\epsilon}_N}(g)- \bar{L}_{\boldsymbol{\epsilon}_N}(g) \vert &\leq \frac{1}{N}\sum_{i=1}^N\left\vert \frac{1}{p_i}-\frac{1}{\pi_i}\right\vert 
\leq 4N\kappa_N/(nd_N),
\end{align}
as soon as $d_N>4$. One may thus directly derive a rate bound for solutions of $\inf_{g\in \mathcal{G}}\widetilde{L}_{\boldsymbol{\epsilon}_N}(g)$ from bound $(ii)$ in Proposition\ref{prop:rate1}.
  In particular, the learning rate achieved by $\bar{g}_N$ is preserved when $1/\sqrt{n}=O(\min_{i\leq N}\pi_i)$ as $N,\; n\rightarrow +\infty$.
\end{remark}
\subsection{Extensions to More General Sampling Schemes}

We now extend the rate bound analysis carried out in the previous subsection
to more complex sampling schemes (described by a random vector $\mathbf{\epsilon}^*_N$ possibly exhibiting a very complex dependence structure).  In order to give an insight into the
arguments which the extension is based on, additional notations are required.
In this section, we consider a general sampling design $R_N^*$ with first order inclusion probabilities $\boldsymbol{\pi}^*_N=(\pi^*_1,\; \ldots,\; \pi^*_N)$ described by
the vector $\boldsymbol{\epsilon}^{*}_{N}=(\epsilon_{1}^{*},\;\ldots,\;
\epsilon_{N}^{*})$ and investigate the performance of minimizers
$\bar{g}^*_{N}$ of the HT empirical risk $\bar{L}%
_{\boldsymbol{\epsilon}^{*}_{N}}(g)=(1/N)\sum_{i=1}^N(\epsilon_i^*/\pi_i^*)\mathbb{I}\{ Y_i\neq g(X_i) \}$ over a class $\mathcal{G}$. We
also consider a rejective sampling scheme $R_N$ described by the r.v.
$\boldsymbol{\epsilon}_{N}$, with first order inclusion probabilities
$\boldsymbol{\pi}_{N}=(\pi_{1},\;\ldots,\; \pi_{N})$ defined on the same probability
space, as well as the following quantity:
\begin{equation}
\label{eq:HTrisk_mix}\check{L}_{\boldsymbol{\epsilon}_{N}}%
(g)=\frac{1}{N}\sum_{i=1}^{N}\frac{\epsilon_{i}}{\pi^*_{i}}%
\mathbb{I}\{Y_{i}\neq g(X_{i})\}
\end{equation}
for any classifier $g$. Observe that $\eqref{eq:HTrisk_mix}$ differs from the HT empirical risk $\bar{L}_{\boldsymbol{\epsilon}_{N}}(g)$ related to the rejective sampling scheme $\boldsymbol{\epsilon}_{N}$ in the weights it involves, the $\pi^*
_{i}$'s instead of the $\pi_{i}$'s namely. Equipped with this notation, the excess of risk of the HT
empirical risk minimizer can be bounded as follows: 
\begin{align}
\label{eq:decomp2}L(\bar{g}^*_{N})-\inf_{g\in\mathcal{G}}L(g)&\leq  2
\sup_{g\in\mathcal{G}}\left\vert L(g)-\widehat{L}_N(g)\right\vert 
+2\sup_{g\in\mathcal{G}}\left\vert \widehat{L}_N(g)-\bar{L}_{\boldsymbol{\epsilon}_{N}%
}(g)\right\vert
\nonumber \\  &+ 2\sup_{g\in\mathcal{G}}\left\vert \bar{L}_{\boldsymbol{\epsilon}_{N}
}(g)-\check{L}_{\boldsymbol{\epsilon}_{N}}(g)\right\vert 
+ 2\sup_{g\in\mathcal{G}}\left\vert \check{L}_{\boldsymbol{\epsilon}_{N}}%
(g)-\bar{L}_{\boldsymbol{\epsilon}^{*}_{N}}(g)\right\vert .
\end{align}
Whereas the first
term on the right hand side of \eqref{eq:decomp2} can be classically controlled using Vapnik-Chervonenkis and McDiarmid inequalities (see \textit{e.g.} \cite{Vapnik}), assertion $(i)$ of Proposition \ref{prop:rate1} provides a control of the second term. Following in the footsteps
of \cite{Hajek64}, the third term shall be bounded by means of a
\textit{coupling} argument, \textit{i.e.} a specific choice of the joint
distribution of $(\boldsymbol{\epsilon}_{N}^{*},\; \boldsymbol{\epsilon}_{N})$
satisfying the distributional margin constraints, while the second term is
controlled by assumptions related to the closeness between the first order
inclusion probabilities $\boldsymbol{\pi}^*_{N}$ and $\boldsymbol{\pi}_{N}$. More precisely,
the assumptions required in the subsequent analysis involve the total
variation distance between the sampling plans $R_N$ and $R_N^*$:
\[
d_{TV}(R_{N},\; R^*_{N})\overset{def}{=}\frac{1}{2}\sum_{s\in \mathcal{P}(\mathcal{I}_N)}\vert R_N(s)-R_N^*(s)\vert.
\]

\begin{theorem}\label{thm:General}
Suppose that Proposition \ref{prop:rate1}'s assumptions are fulfilled. Set $\kappa_N^*=(n/N)\min_{i\leq N}\pi^*_i$ and $\kappa_N=(n/N)/\min_{i\leq N}\pi_i$. Then, there exists a universal constant $C<+\infty$ such that we have, $\forall N\geq 1$,
\begin{align}\label{eq:General_rate_up}
\mathbb{E}\left[ L(\bar{g}^*_N)-\inf_{g\in \mathcal{G}}L(g)\right] &\leq
2\sqrt{2\kappa_N\frac{V\log (N+1)}{n}}
+ 4\kappa_N\frac{V\log (N+1)}{3n} \nonumber\\
&+C\sqrt{\frac{V}{N}}
+ 2(\kappa_N^*+\kappa_N) (N/ n)d_{TV}(R_N,\; R_N^*) ,
\end{align}
where the infimum is taken over the set of rejective schemes $R_N$ with first order inclusion probabilities $\boldsymbol{\pi}_N=(\pi_1,\;\ldots,\; \pi_N)$.
\end{theorem}
The proof is given in the Supplementary Material. The rate bound obtained depends on the minimum error made when approximating the sampling plan by a rejective sampling plan in terms of total variation distance. In practice,
following in the footsteps of \cite{Hajek64} or \cite{Ber98}, it can be controlled by exhibiting a
specific coupling $(\boldsymbol{\epsilon}^{*}_{N},\; \boldsymbol{\epsilon}%
_{N})$. One may refer to \cite{Ber98} for many coupling results of this nature, in particular when the approximating scheme $\boldsymbol{\epsilon}_N$ is of rejective type. 

\section{Illustrative Numerical Experiments}\label{sec:num}

In this section we display numerical experiments to illustrate the relevance of HT risk minimization.
We first consider the case where $g(X)=sign(k(X)^T\theta+b)$, where $k$ is some mapping function, $T$ denotes the transposition operator, $\theta$, $b$ are some parameters. As mentionned in \ref{remarkSurrogate}, we consider the hinge loss as a convex surrogate of the $0-1$ loss and add some $l_2$ regularization term. This leads to the "Weighted SVM" formulation below:
\begin{align*}
\min_{\theta,b}\frac{1}{N}\sum_{i\in S}\frac{1}{\pi_i} \max(0,1-Y_i(k(X_i)^T\theta-b))+\lambda \Vert \theta \Vert^2.
\end{align*}
We use the  gaussian r.b.f kernel and perform cross validation to appropriately choose the value of $\lambda$.
We then consider the task of learning classification trees using the CART algorithm.
These classifiers are trained using the scikit-learn library \cite{scikit-learn} and, we account for the randomness of our experiments by shuffling our datasets and repeating the experiments 50 times.

We first generate a two class dataset $\mathcal{D}$ in $\mathbb{R}^{10}$ of size $20000$ by sampling independent observations from two multivariate normal distribution. A similar dataset $\mathcal{D}_{test}$ of size 2000 is generated to test our classifiers. Denoting by $I_{d}$ the identity matrix in $\mathbb{R}^d$, the positive class has mean $(0,\dots,0)$ and covariance matrix equal to $I_{10}$, the negative class has mean $(1,\dots,1)$ and covariance matrix equal to $10 \times I_{10}$. We then build a dataset $\widetilde{\mathcal{D}}$ of size $1100$ via a rejective sampling scheme applied to $\mathcal{D}$. Observations from the negative class being more noisy we assign them first order probability equal to $0.1$, and assign first order probability equal to $0.01$ to observation from the positive class. To allow for a fair comparison, we also build a dataset $\widehat{\mathcal{D}}$ of size $1100$ by sampling without replacement within $\mathcal{D}$.
We then learn the different classifiers on $\widetilde{\mathcal{D}}$ and $\widehat{\mathcal{D}}$, and display the results in Table\ref{table:datasets}.

\begin{table}[h]
 \label{NumericalResults}
\begin{center}
\begin{tabular}{lccccc}
& {\bf Mean}  &{\bf Std Deviation} \\
\hline \\
\textbf{Weighted SVM on $\widetilde{\mathcal{D}}$} & 0.02 & 0.005  \\
\textbf{Unweighted SVM on $\widetilde{\mathcal{D}}$} & 0.18  & 0.02 \\
\textbf{SVM on $\widehat{\mathcal{D}}$} & 0.04 & 0.005  \\
\textbf{Weighted CART on $\widetilde{\mathcal{D}}$} &  0.06  & 0.01 \\
\textbf{Unweighted CART on $\widetilde{\mathcal{D}}$} & 0.11  & 0.03 \\
\textbf{CART on $\widehat{\mathcal{D}}$} & 0.08 & 0.01  \\
\end{tabular}
\end{center}
\caption{\textit{Average over 50 runs of the prediction error on $\mathcal{D}_{test}$ and its standard deviation.}}
\end{table}

Overall, taking into accounts the inclusion probability allows to consider a training set of reduced size and therefore reduce the computationnal complexity of the learning procedure without damaging the quality of the prediction . Similar experiments on real datasets are displayed in the Supplementary Material for which similar conclusions hold.

\section{Conclusion}\label{sec:conclusion}
Most theoretical studies providing a statistical explanation for the success of learning algorithms based on the ERM paradigm fully ignore the possible impact of the sampling scheme producing the training data and stipulate that observations are independent replications of a generic r.v. or are uniformly sampled without replacement in a larger dataset. Through the  generalizable example of rejective sampling, this paper shows that such studies can be extended to situations where training data are obtained by more general sampling schemes and possibly exhibit a complex dependence structure, provided that related probablity weights are appropriately incorporated in the risk functional.
 
\section*{Appendix}

\subsection*{Proof of Theorem \ref{thm:neg}}
Considering the usual representation of the distribution of $(\epsilon_1,\; \ldots,\; \epsilon_N)$ as the conditional distribution of a sample of independent Bernoulli variables $(\epsilon^*_1,\; \ldots,\; \epsilon^*_N)$ conditioned upon the event $\sum_{i=1}^N\epsilon^*_i=n$ (see subsection \ref{subsec:survey}), the result is a consequence of Theorem 2.8 in \cite{JDP83}. 

\subsection*{Bernstein inequality for sums of negatively associated random variables}
 
For simplicity, we first establish the following tail bound for negatively associated random variables, which extends the usual Bernstein inequality in the i.i.d. setting, see \cite{Bernstein}. Proofs of Proposition \ref{prop:rate1} and Theorem \ref{thm:General} are then deduced from Theorem \ref{thm:neg} and Theorem \ref{thm:BernNeg} (see Supplementary Material) .
  \begin{theorem}\label{thm:BernNeg}
  Let $Z_1,\; \ldots,\; Z_N$ be negatively associated real valued random variables such that $\vert Z_i\vert \leq c<+\infty$ a.s. $\mathbb{E}[Z_i]=0$ and $\mathbb{E}[Z_i^2 ]\leqslant\sigma_i^2$ for $1\leq i \leq N$. Then, for all $t>0$, we have: $\forall N\geq 1$,
  $$
  \mathbb{P}\left\{\sum_{i=1}^NZ_i  \geq t\right\}\leq \exp\left( -\frac{t^2}{\frac{2}{3}ct+2\sum_{i=1}^N\sigma^2_i} \right).
  $$
  \end{theorem}
 
 Before detailing the proof, observe that a similar bound holds true for the tail probability $ \mathbb{P}\left(\sum_{i=1}^NZ_i  \leq - t\right)$ (and for $\mathbb{P}\left(\vert \sum_{i=1}^NZ_i \vert \geq  t\right)$ as well, up to a multiplicative factor $2$). Refer also to Theorem 4 in \cite{Janson} for a similar result in a more restrictive setting (\textit{i.e.} for tail bounds related to sums of negatively associated r.v.'s).
 \begin{proof}
 The proof starts off with the usual Chernoff method: for all $\lambda>0$,
 \begin{equation}\label{eq:Chernoff}
  \mathbb{P}\left\{\sum_{i=1}^NZ_i  \geq t\right\}\leq \exp\left( -t\lambda +\log \mathbb{E}\left[e^{t\sum_{i=1}^N Z_i} \right] \right).
 \end{equation}
 Next, observe that, for all $t>0$, we have 
 \begin{align*}\label{eq:neg2}
 \mathbb{E}\left[e^{t\sum_{i=1}^nZ_i}\right]&=\mathbb{E}\left[e^{tZ_n}e^{t\sum_{i=1}^{n-1}Z_i}\right]\\
 &\leq \mathbb{E}\left[e^{tZ_n} \right]\mathbb{E}\left[e^{t\sum_{i=1}^{n-1}Z_i}  \right]\\
 & \leq  \prod_{i=1}^n\mathbb{E}\left[ e^{tZ_i} \right],
 \end{align*}
 using the property \eqref{eq:neg} combined with a descending recurrence on $i$. The proof is finished by plugging \eqref{eq:neg2} into \eqref{eq:Chernoff}, using an adequate control of the log-Laplace transform of the $Z_i$'s and optimizing finally the resulting bound w.r.t. $\lambda>0$, just like in the proof of the classic Bernstein inequality, see \cite{Bernstein}.
 \end{proof}

\section*{Supplementary - Proof of Proposition \ref{prop:rate1} }

We start off by writing $S:=\sup_{g\in \mathcal{G}}\left\vert \bar{L}_{\boldsymbol{\epsilon}_N}(g)-\widehat{L}_N(g) \right\vert$ as $\sup_{g\in \mathcal{G}}\vert \frac{1}{N}\sum_{i=1}^{N} \left(\frac{\epsilon_{i}}{\pi_{i}}-1\right) \, \mathbb{I}\left\{
       g(X_{i}) \neq Y_{i} \right\} \vert $
and apply Theorem \ref{thm:neg} conditionnaly upon $\mathcal{D}_N$ to the r.v $Z_i :=\frac{1}{N}\left(\frac{\epsilon_{i}}{\pi_{i}}-1\right) \, \mathbb{I}\left\{g(X_{i}) \neq Y_{i} \right\}$ . Indeed, the $(\pi_{i})_{i=1}^N$ and $(\mathbb{I}\left\{g(X_{i}) \neq Y_{i} \right\})_{i=1}^N$ being positive real numbers, Theorem \ref{thm:neg} altogether with \cite{JDP83} implies that $(Z_i)_{i=1}^N$ are negatively associated. Since $\vert Z_i \vert \leqslant \frac{1}{N}\max(1,\frac{1}{\pi_i}-1)\leqslant\frac{1}{N \pi_i} \leqslant \frac{\kappa_{N}}{n} $ and $\mathbb{E}[Z_i^2]\leqslant \frac{1-pi_i}{N^2\pi_i}\leqslant\frac{\kappa_{N}}{nN}$ we have :
\begin{align*}
\mathbb{P}\left\{\sum_{i=1}^NZ_i  \geq t \vert \mathcal{D}_N \right\}\leq \exp\left( -\frac{nt^2}{\frac{2}{3}\kappa_N t+2\kappa_N}\right).
\end{align*}
Applying the same method to the r.v $(-Z_i)_{i=1}^N$ and taking the union bound yields :
\begin{align*}
\mathbb{P}\left\{ \vert \sum_{i=1}^N Z_i \vert  \geq t \vert \mathcal{D}_N \right\}\leq 2\exp\left( -\frac{nt^2}{\frac{2}{3}\kappa_N t+2\kappa_N}\right).
\end{align*}
By virtue of Sauer's lemma, since the class $\mathcal{G}$ has finite $VC$-dimension $V$, we have by taking expectation w.r.t $\mathcal{D}_N$:
\begin{align*}
\mathbb{P}\left\{ S   \geq t \right\}\leq 2(N+1)^V\exp\left( -\frac{nt^2}{\frac{2}{3}\kappa_N t+2\kappa_N}\right).
\end{align*}
The high probability bound is then easily deduced by choosing  $\delta=2(N+1)^V\exp\left( -\frac{nt^2}{\frac{2}{3}\kappa_N t+2\kappa_N}\right)$ so that :
\begin{align*}
\left(t-\frac{\log(\frac{2}{\delta})+V\log(N+1)}{3n}\kappa_N \right)&=\left( \frac{\log(\frac{2}{\delta})+V\log(N+1)}{3n}\kappa_N\right)^2\\
&+\frac{2\left(\log(\frac{2}{\delta})+V\log(N+1)\right)}{n}\kappa_N
\end{align*}
leading to the following upperbound :
\begin{align*}
t\leqslant \frac{2\kappa_N\left(\log(\frac{2}{\delta})+V\log(N+1)\right)}{3n}+\sqrt{2\frac{\log(\frac{2}{\delta})+V\log(N+1)}{n}\kappa_N}.
\end{align*}
The second claim of Proposition \ref{prop:rate1} is established using 
\begin{multline}
L(\bar{g}_N)-L^*\leq \inf_{g\in\mathcal{G}}L(g)-L^* +2
\sup_{g\in\mathcal{G}}\left\vert L(g)-\widehat{L}_N(g)\right\vert 
+2\sup_{g\in\mathcal{G}}\left\vert \widehat{L}_N(g)-\bar{L}_{\boldsymbol{\epsilon}_{N}%
}(g)\right\vert,
\end{multline}
altogether with classical results on ERM applied to the term $\sup_{g\in\mathcal{G}}\left\vert L(g)-\widehat{L}_N(g)\right\vert$ and a union bound.
\section*{Supplementary - Proof of Theorem\ref{thm:General} }

Starting from \eqref{eq:decomp2}, we only have to derive bounds for the quantities $S_1:=\sup_{g\in\mathcal{G}}\left\vert \bar{L}_{\boldsymbol{\epsilon}_{N}}(g)-\check{L}_{\boldsymbol{\epsilon}_{N}}(g)\right\vert$ and $S_2:=\sup_{g\in\mathcal{G}}\left\vert \check{L}_{\boldsymbol{\epsilon}_{N}}(g)-\bar{L}_{\boldsymbol{\epsilon}^{*}_{N}}(g)\right\vert$. Starting woth the first one, we have :
\begin{align*}
S_1&=\sup_{g\in\mathcal{G}}\vert  \frac{1}{N} \sum_{i=1}^N \epsilon_i \left( \frac{1}{\pi_i^*}-\frac{1}{\pi_i} \right)  \mathbb{I}\left\{g(X_{i}) \neq Y_{i} \right\}\vert\\
&\leqslant \frac{1}{N} \sum_{i=1}^N \epsilon_i \vert \frac{1}{\pi_i^*}-\frac{1}{\pi_i}  \vert
\end{align*}
so that taking expectation w.r.t $\epsilon_N$ conditionned upon $\mathcal{D}_N$ yields :
\begin{align*}
\mathbb{E}[S_1\vert \mathcal{D}_N]&\leqslant \frac{1}{N} \sum_{i=1}^N  \vert \frac{\pi_i^*-\pi_i}{\pi_i} \vert \\
&\leqslant \kappa_N \frac{N}{n}\frac{1}{N}\sum_{i=1}^N \vert \pi_i -\pi_i^* \vert\\
&\leqslant \kappa_N \frac{N}{n} d_{TV}(R_{N},\; R^*_{N}),
\end{align*}
taking  expectation w.r.t $\mathcal{D_N}$ gives an upperbound on $S_1$. We now turn to the analysis of $S_2$  which is very similar :
\begin{align*}
S_2&=\sup_{g\in\mathcal{G}}\vert  \frac{1}{N} \sum_{i=1}^N \frac{\epsilon_i^* -\epsilon_i}{\pi_i^*} \mathbb{I}\left\{g(X_{i}) \neq Y_{i} \right\}\vert\\
&\leqslant \frac{1}{N} \sum_{i=1}^N  \frac{ \vert \epsilon_i^* -\epsilon_i \vert}{\pi_i^*}.
\end{align*}
We then take expectation conditionned upon $\mathcal{D}_N$ and easily obtain
\begin{align*}
\mathbb{E}[S_2\vert \mathcal{D}_N] \leqslant \kappa_N^* \frac{N}{n} d_{TV}(R_{N},\; R^*_{N})
\end{align*}
which conclude the proof.
\section*{Supplementary - On biased HT risk minimization}
Eq. \eqref{eq:bias} directly results from the following lemma.

\begin{lemma}\label{lem:bias}
Suppose that $d_{N}>4$. We have, for all $i\in\{1,\; \ldots,\; N  \}$, 
\[
\left\vert1/\pi_{i}-1/p_{i}\right\vert \leq\frac{4}{d_{N}}\times(1-\pi_{i})/\pi_{i}.
\]
\end{lemma}

\begin{proof}
The proof follows from the representation (5.14) on p1509 in \cite{Hajek64}. Denote by $P_N$ a Poisson sampling distribution on $\mathcal{I}_N$ with inclusion probabilities $p_1,\; \ldots,\; p_N$, the canonical parameters of $R_N$. For all $i\in \{1,\; \ldots,\; N\}$, we have:
\begin{align*}
&\frac{\pi_{i}}{p_{i}}\frac{1-p_{i}}{1-\pi_{i}}   =\left(  \sum_{s\in \mathcal{P}(\mathcal{I}_N):\; i\in
\mathcal{I}_N\setminus\{s \}}P(s)\right)  ^{-1}\\
&\times\sum_{s\in \mathcal{P}(\mathcal{I}_N):\;  i\in \mathcal{I}_N\setminus\{s \}}P(s)\sum_{h\in s}\frac{1-p_{h}}%
{\sum_{j\in s}(1-p_{j})+(p_{h}-p_{i})}\\
& =\left(  \sum_{s:\ i\in \mathcal{I}_N\setminus\{s \}}P_N(s)\right)  ^{-1}\\
&\times\sum_{s:\ i\in \mathcal{I}_N\setminus\{s \}}%
P_N(s)\sum_{h\in s}\frac{1-p_{h}}{\sum_{j\in s}(1-p_{j})\left(  1+\frac
{(p_{h}-p_{i})}{\sum_{j\in s}(1-p_{j})}\right)  }%
\end{align*}
Now recall that for any $x\in]-1/2,1[ $, we have:
\[
1-x\leq\frac{1}{1+x}\leq1-x+2x^{2}.
\]
It follows that
\begin{align*}
&\frac{\pi_{i}}{p_{i}}\frac{1-p_{i}}{1-\pi_{i}}   \leq1-\left(  \sum_{s:\ i\in
\mathcal{I}_N\setminus\{s \}}P(s)\right)  ^{-1}\\
&\times\sum_{s:\ i\in \mathcal{I}_N\setminus\{s \}}P(s)\sum_{h\in s}\frac{(1-p_{h}%
)(p_{h}-p_{i})}{\left(  \sum_{j\in s}(1-p_{j})\right)  ^{2}}\\
& +2\left(  \sum_{s:\ i\in \mathcal{I}_N\setminus\{s \}}P(s)\right)  ^{-1}\sum_{s:\ i\in \mathcal{I}_N\setminus\{s \}}%
P(s)\sum_{h\in s}\frac{(1-p_{h})(p_{h}-p_{i})^{2}}{\left(  \sum_{j\in
s}(1-p_{j})\right)  ^{3}}%
\end{align*}
Following now line by line the proof on p. 1510 in \cite{Hajek64} and noticing that $\sum_{j\in s}%
(1-p_{j})\geq d_N/2>2$ (see Lemma 2.2 in \cite{Hajek64}), we have, as soon as $d_N>4$,
\begin{align*}
\left\vert \sum_{h\in s}\frac{(1-p_{h})(p_{h}-p_{i})}{\left(  \sum_{j\in
s}(1-p_{j})\right)  ^{2}}\right\vert  & \leq\frac{1}{\left(  \sum_{j\in
s}(1-p_{j})\right)  } \leq\frac{2}{d_{N}}
\end{align*}
and similarly%
\begin{align*}
\sum_{h\in s}\frac{(1-p_{h})(p_{h}-p_{i})^{2}}{\left(  \sum_{j\in s}%
(1-p_{j})\right)  ^{3}}  & \leq\frac{1}{\left(  \sum_{j\in s}(1-p_{j})\right)
^{2}} \leq\frac{4}{d_{N}^{2}}.
\end{align*}
This yieds: $\forall i\in\{1,\; \ldots,\; N  \}$,
\[
1-\frac{2}{d_{N}}\leq\frac{\pi_{i}}{p_{i}}\frac{1-p_{i}}{1-\pi_{i}}\leq
1+\frac{2}{d_{N}}+\frac{8}{d_{N}^{2}}%
\]
and
\[
p_{i}(1-\pi_{i})(1-\frac{2}{d_{N}})\leq\pi_{i}(1-p_{i})\leq p_{i}(1-\pi
_{i})(1+\frac{2}{d_{N}}+\frac{8}{d_{N}^{2}}),
\]
leading then to
\[
-\frac{2}{d_{N}}(1-\pi_{i})p_{i}\leq\pi_{i}-p_{i}\leq p_{i}(1-\pi_{i}%
)(\frac{2}{d_{N}}+\frac{8}{d_{N}^{2}})
\]
and finally to
\[
-\frac{(1-\pi_{i})}{\pi_{i}}\frac{2}{d_{N}}\leq\frac{1}{p_{i}}-\frac{1}%
{\pi_{i}}\leq\frac{(1-\pi_{i})}{\pi_{i}}(\frac{2}{d_{N}}+\frac{8}{d_{N}^{2}}).
\]
Since $1/d_{N}^{2}\leq 1/d_{N}$ as soon as $d_{N}\geq 1$, the lemma is proved. 
\end{proof}

\section*{Supplementary - Further details on the rejective scheme}

Let $n\leq N$ and consider a vector $\boldsymbol{\pi}= (\pi_{1},\,\ldots,\,\pi_{N})$ of first order inclusion probabilities. Further define $\mathcal{S}_n := \{s \in \mathcal{P}(\mathcal{I}_N) : \#s = n\}$, the set of all samples in population $\mathcal{I}_N$ with cardinality $n$. The rejective sampling \cite{Hajek64, Ber98}, sometimes called conditional Poisson sampling, exponential design without replacement or maximum entropy design, is the sampling design $R_{N}$ that selects samples of fixed size $n(s) = n$ so as to maximize the entropy measure
$H(R_{N})= - \sum_{s \in \mathcal{S}_n}R_{N}(s)\, \log R_{N}(s)$,
subject to the constraint that its vector of first order inclusion probabilities coincides with $\boldsymbol{\pi}$. It is easily implemented in two steps:

\begin{enumerate}
\item Draw a sample $S$ according to a Poisson plan $P_{N}$, with properly chosen first order inclusion probabilities $\mathbf{p}_N= (p_{1},\; \ldots,\; p_{N})$. The representation is called canonical if $\sum_{i = 1}^N p_{i} = n$. In that case, relationships between each $p_{i}$ and $\pi_{i}$, $1\leq i \leq N$, are established in \cite{Hajek64}.

\item If $n(S)\neq n$, then reject sample $S$ and go back to step one, otherwise stop.
\end{enumerate}
Vector $\mathbf{p}$ must be chosen in a way that the resulting first order inclusion probabilities coincide with $\boldsymbol{\pi}$, by means of a dedicated optimization algorithm \cite{tille2006sampling}. The corresponding probability distribution is given for all $s \in \mathcal{P}(\mathcal{I}_{N})$ by
$R_{N}(s) = \frac{P_{N}(s)\,\mathbb{I}
\{\#s = n\}  }{\sum_{s^\prime \in \mathcal{S}_n}P_{N}(s^{\prime})} \ \propto \ \prod_{i\in s}p_{i}\prod_{i\notin s}(1-p_{i}%
)\times\mathbb{I}\{  \#s=n\}$,
where $\propto$ denotes the proportionality.

\section*{Supplementary - Stratified sampling}
A stratified sampling design permits to draw a sample $S$ of fixed size $n(S)=n\leq N$ within a population $\mathcal{I}_{N}$ that can be partitioned into $K\geq1$ distinct strata $\mathcal{I}_{N_{1}},\dots,\mathcal{I}_{N_{K}}$ (known a priori) of respective sizes $N_{1},\,\ldots,\,N_{K}$ adding up to $N$. Let $n_{1},\dots,n_{K}$ be non-negative integers such that $n_1 + \dots + n_K = n$, then the drawing procedure is implemented in $K$ steps: within each
stratum $\mathcal{I}_{N_{k}}$, $k\in\{1,\ldots,K\}$, perform a SWOR of size $n_{k}\leq N_{k}$ yielding a sample $S_{k}$. The final sample is obtained by assembling these sub-samples: $S=\bigcup_{k=1}^{K}S_{k}$. The probability of drawing a specific sample $s$ by means of this survey design is
$R_{N}^{\text{str}}(s)=\sum_{k=1}^{K} \dbinom{N_k}{n_k}^{-1}$.
Naturally, first and second order inclusion probabilities depend on the
stratum to which each unit belong: for all $i \neq j$ in $\mathcal{U}_N$,
$\pi_{i}(R^{\text{str}}_N) = \sum_{k=1}^{K}\frac{n_{k}}{N_{k}}\,\mathbb{I}\{
i\in\mathcal{U}_{N_{k}}\}$
 and $ \pi_{i,j}(R^{\text{str}}_N) = \sum_{k=1}^{K} \frac{n_{k}(n_{k}-1)}{N_{k}(N_{k}-1)}\, \mathbb{I}\{ (i,j) \in \mathcal{U}_{N_{k}}^{2}\}$.

\section*{Supplementary - Rao-Sampford sampling}
The Rao-Sampford sampling design generates samples $s \in \mathcal{P}(\mathcal{I}_{N})$ of fixed size $n(s)=n$ with respect to some given first order inclusion probabilities $\boldsymbol{\pi}^{RS} := (\pi_{1}^{RS},\ldots,\pi_{N}^{RS})$, fulfilling the condition $\sum_{i=1}^{N} \pi_{i}^{RS} = n$, with probability 
\begin{equation*}
R_{N}^{RS}(s)=\eta\sum_{i\in s}\pi_{i}^{RS}\prod
_{j\notin s}\frac{\pi_{j}^{RS}}{1-\pi_{j}^{RS}}.
\end{equation*}
Here, $\eta > 0$ is chosen such that $\sum_{s \in \mathcal{P}(\mathcal{I}_{N})}R_{N}^{RS}(s)=1$. In practice, the following algorithm is often used to implement such a design \cite{Ber98}:

\begin{enumerate}
\item select the first unit $i$ with probability $\pi_{i}^{RS} / n$,
\item select the remaining $n - 1$ units $j$ with drawing probabilities
proportional to $\pi_{j}^{RS}/(1 - \pi_{j}^{RS})$, $j = 1,\,\ldots,\, N$,
\item accept the sample if the units drawn are all distinct, otherwise reject
it and go back to step one.
\end{enumerate}

\section*{Additional experiments}

We consider the following datasets which were obtained via a stratified sampling design. We point out that this sampling scheme involves \textit{negatively associated} (binary) random variables so that the theoretical results obtained in the rejective case extend to this scheme.

\begin{table}[h]
\label{tab:datasets}
\begin{center}
\begin{tabular}{lcc}
&\textbf{N} & \textbf{Number of features} \\
\hline \\
\textbf{incaIndiv} & 4079  & 326 \\
\textbf{GJB} & 2001  & 130 \\
\textbf{privacy3} & 316  & 95 \\
\textbf{privacy4} & 301  &124 \\
\end{tabular}
\end{center}
\end{table}

The dataset \textit{incaIndiv} \footnote{\url{https://https://www.data.gouv.fr/fr/datasets/}} contain informations on the food consumption of the french population. The dataset \textit{GJB}\footnote{\url{http://www.pewinternet.org/datasets/june-10-july-12-2015-gaming-jobs-and-broadband/}} contains questions about job seeking and the internet, workforce automation, online dating and smartphone use among Americans. The datasets \textit{privacy3}\footnote{\url{http://www.pewinternet.org/datasets/nov-26-2014-jan-3-2015-privacy-panel-3/}} and \textit{privacy4}\footnote{\url{http://www.pewinternet.org/datasets/jan-27-feb-16-2015-privacy-panel-4/}} contain questions about privacy and information sharing. On the datasets \textit{incaIndiv} and \textit{incaCompl} we try to predict whether or not someone is an adult, on the dataset \textit{GJB} we will try to learn to predict the gender, and on the datasets \textit{privacy3} and \textit{privacy4} we will predict an answer to some questions among 5 possibilities.

We perform our experiments by randomly splitting the datasets \textit{incaIndiv}, \textit{incaCompl},\textit{GJB} into a training set (roughly $70$ percent of the initial dataset) and a test set. The size of \textit{privacy3} and \textit{privacy4} being much smaller we perform 10-fold cross-validation.  

\begin{table}[h]\label{NumericalResults}
\label{tab:datasets}
\begin{center}
\begin{tabular}{lccccc}
 & \textbf{incaIndiv}  & \textbf{GJB} & \textbf{privacy3} & \textbf{privacy4}\\
\hline \\
\textbf{Weighted SVM} & 0.16 & 0.36 & 0.46 & 0.48 \\
\textbf{Unweighted SVM} & 0.19  & 0.43 & 0.50 & 0.52\\
\textbf{Weighted CART} & 0.04   & 0.41 & 0.49  & 0.54 \\
\textbf{Unweighted CART} & 0.05  & 0.43 & 0.52 & 0.57\\
\end{tabular}
\end{center}
\caption{\textit{Average over 50 runs of the prediction error}}
\end{table}

\bibliographystyle{plain}
\bibliography{ERM_sampling}

\end{document}